\documentclass[10pt]{amsart}
\usepackage{latexsym, amsfonts, delarray, amsmath,delarray}

\newtheorem{thm}{Theorem}[section]

\newtheorem{prop}[thm]{Proposition}

\newtheorem{lem}[thm]{Lemma}

\newtheorem{remark}[thm]{Remark}

\newtheorem{quest}[thm]{Question}

\def \R {{\rm{I\hspace{-1mm}R}}}

\def \bx {{\bf x}}

\def\proof{\noindent \textsc{Proof:}\ }
\def\endproof{$\Box$ \medskip}

\newcommand{\calL}{{\mathcal L}}

\def\calP{\mathcal{P}}
\def\calO{\mathcal{O}}

\def\FF{\mathbb{F}}

\def\RR{\mathbb{R}}
\def\real{\mathbb{R}}
\def\ZZ{\mathbb{Z}}
\def\zed{\mathbb{Z}}

\def\mod{\mathrm{\; mod\; }}

\def\SL2{\mathrm{SL}_{2}}

\def\proof{\noindent \textsc{Proof:}\ }
\def\endproof{ \hfill \parbox{0.5cm}{$\Box$}}

\newcommand{\rmv}[1]{}

\def \spn {\mathrm{span}}
\def \span {\mathrm{span_{K}}}

 \def \A {{\bf A}}
 \def \P {{\bf P}}
\def \Q {{\bf Q}}
\def \R {{\bf R}}
\def \S {{\bf S}}
\def \T {{\bf T}}
\def \U {{\bf U}}

\begin{document}
\title{Lattices from elliptic curves over finite fields}
\author{Lenny Fukshansky}
\author{Hiren Maharaj}\thanks{The first author was partially supported by NSA Young Investigator Grant \#1210223 and Simons Foundation grants \#208969, 279155.}
\address{Department of Mathematics, 850 Columbia Avenue, Claremont McKenna College, Claremont, CA 91711}
\email{lenny@cmc.edu}
\address{8543 Hillside Road, Rancho Cucamonga, CA 91701}
\email{hmahara@g.clemson.edu}


\begin{abstract}
In their well known book~\cite{tv} Tsfasman and Vladut introduced a construction of a family of {\it function field lattices} from algebraic curves over finite fields, which have asymptotically good packing density in high dimensions. In this paper we study geometric properties of lattices from this construction applied to elliptic curves. In particular, we determine the generating sets, conditions for 
well-roundedness and a formula for the number of minimal vectors. We also prove a bound on the covering radii of these lattices, which improves on the standard inequalities.
\end{abstract}

\maketitle

\section{Introduction}
\label{intro}

Let $L \subset \RR^n$ be a lattice of rank $k \leq n$, and let $V = \spn_{\RR} L$ be the $k$-dimensional subspace of~$\RR^n$ spanned by~$L$. The {\it minimum distance} of $L$ is
$$d(L) = \min \{ \|\bx\| : \bx \in L \},$$
where $\|\ \|$ is the usual Euclidean norm in~$\RR^n$. The {\it lattice (sphere) packing} in~$V$ associated to~$L$ is the arrangement of balls of radius~$d(L)/2$ centered at points of~$L$, and the {\it density} $\Delta(L)$ of such packing is the proportion of~$V$ taken up by this arrangement,~i.e.
\begin{equation}
\label{delta}
\Delta(L) =  \frac{\omega_k d(L)^k}{2^k \det L},
\end{equation}
where $\omega_k = \frac{\pi^{ \frac{k}{2}}}{\Gamma(\frac{k}{2}+1)}$ is the volume of a $k$-dimensional unit ball. Given a $k$-dimensional subspace $V$ of $\real^n$, the lattice packing problem in~$V$ is to find a lattice~$L \subset V$ such that $V = \spn_{\RR} L$ and $\Delta(L)$ is maximal among all lattice packing densities in~$V$. It is easy to see that the lattice packing density problem in~$V$ is equivalent to this problem in~$\real^k$, where we denote the maximal lattice packing density achieved by~$\Delta_k$. The values of~$\Delta_n$ are currently only known for dimensions $1 \leq n \leq 8$~\cite{conway} and $n=24$~\cite{cohn_kumar} with explicit constructions of lattices achieving these densities. More generally, the famous Minkowski-Hlawka theorem states that in every dimension~$n$ there exists a lattice whose packing density is $\geq \zeta(n)/2^{n-1}$, where $\zeta$ stands for the Riemann zeta-function. Unfortunately, the known proofs of Minkowski-Hlawka theorem are non-constructive, and for arbitrary dimensions constructions of lattices satisfying this bound are not known. On the other hand, the mere existence of this bound motivated various constructions of asymptotic families of lattices, one in every dimension, whose packing density comes as close as possible to Minkowski-Hlawka. One such family, which produces particularly nice results as~$n \to \infty$ are the so-called {\it function field lattices}, constructed by Tsfasman and Vladut (see~\cite{tv}, pp.~578--583).

We use notation of~\cite{stich}. The construction of function field lattices given in \cite{tv} is as follows. Let $F$ be an algebraic function field (of a single variable) with the finite field $\FF_q$ as its full field of constants. Let $\calP = \{ P_0,P_1,P_2,\ldots, P_{n-1}\}$ be the set of rational 
places of $F$. Corresponding to each place $P_i$, let $v_i$ denote the corresponding normalized discrete valuation and let  $\calO_\calP^*$ be the set of all nonzero functions  $f\in F$ whose divisor has support contained in the set $\calP$. Then $\calO_\calP^*$ is an abelian group, $\sum_{i=0}^{n-1} v_i(f) = 0$ for each $f \in \calO_\calP^*$, and we let
$$\deg f = \sum_{v_i(f) > 0} v_i(f) = \frac{1}{2} \sum_{i=0}^{n-1} |v_i(f)|.$$
Define the homomorphism
$\phi_\calP : \calO_\calP^* \to \ZZ^n$ (here $n= |\calP|$, the number of rational places of $F$) by $$\phi_\calP(f) = (v_0(f),v_1(f), \ldots, v_{n-1}(f)).$$ 
Then $L_\calP := \mathrm{Image}(\phi_\calP)$ is a finite-index sublattice of the root lattice
$$A_{n-1} = \left\{ \bx \in \zed^n : \sum_{i=0}^{n-1} x_i = 0 \right\}$$
with minimum distance 
\begin{equation} \label{mind}
d(L_\calP) \ge \min \left\{  \sqrt{2\deg f} : f \in \calO_\calP^*\setminus\FF_q \right\},
\end{equation} 
and 
\begin{equation} \label{det}
\det L_\calP \le \sqrt{n} h_F \le \sqrt{n} \left ( 1 + q  + \frac{n-q-1}{g} \right )^g,
\end{equation}
where $g$ is the genus of $F$ and $h_F$ is the divisor class number of $F$, that is, the size of the  group of divisor classes of $F$ of  degree 0, denoted by $\mathrm{Cl}^0(F)$. Here, as in \cite{tv}, we can identify $\ZZ^n$ with the set of all divisors with support in $\calP$ and $A_{n-1}$ with all such divisors of degree~$0$. We will often make use of this identification when working with lattice vectors by working with the corresponding divisors instead.

Equation~\eqref{delta} above indicates that to maximize the packing density one should take a lattice with the quotient of minimum distance to the determinant as large as possible. In the Tsfasman-Vladut construction above, this can be achieved when the quotient $n/g$ is relatively large, as indicated in~\cite{tv}. In particular, Tsfasman and Vladut consider families of curves  for which the packing density of  the corresponding lattices is asymptotically good as~$n$ grows. On the other hand, it is well known (see, for instance~\cite{martinet}) that lattices in~$\real^n$ with particularly high packing density are usually {\it well-rounded}, i.e., their sets of minimal nonzero vectors (with respect to Euclidean norm) contain $n$ linearly independent ones. This observation prompted us to ask the following natural question.

\begin{quest} \label{Q1} For which algebraic function fields $F$ is the corresponding lattice $L_{\calP}$ well-rounded?
\end{quest}

\noindent
The main goal of this note is to provide the following partial answer to this question.

\begin{thm} \label{main} Let $F$ be an algebraic function field over $\FF_q$ with $g=1$ and $n \geq 5$. Then the corresponding lattice $L_{\calP}$ is generated by its minimal vectors, and hence well-rounded.
\end{thm}

We also investigate a variety of geometric properties of the lattice $L_{\calP}$ when $F$ has genus 1, i.e., when the underlying curve is elliptic, in particular establishing formulas for the minimum distance (Lemma~\ref{form}) and the number of minimal vectors (Theorem~\ref{min_num}) of $L_{\calP}$, and conclude with a non-trivial bound on the covering radii of such lattices (Theorem~\ref{cover}). In Section~\ref{prelim} we set the notation and prove several preliminary lemmas on elliptic curves and corresponding function fields, in particular obtaining an explicit description for a generating set of the lattice~$L_{\calP}$ in the case of elliptic curves (Theorem~\ref{gen}). We then prove our main results in Section~\ref{elliptic}. We are now ready to proceed.
\bigskip

\section{Notation and preliminary results}
\label{prelim}

In this section we establish some necessary preliminaries on elliptic curves. An elliptic curve is a pair $(E,O)$, where $E$ is a curve of genus 1 and $O\in E$. In this paper, the elliptic curves are always defined over a finite field $K: = \mathbb{F}_q$. 
 It can be shown \cite[Proposition 6.1.2]{stich} that if $\mathrm{char}(K) \ne 2$ then $K(E) = K(x,y)$, where $y^2 = f(x)$ and $f(x) \in K[x]$ is square-free of degree three, and if $\mathrm{char}(K) =2$ then $K(E)=K(x,y)$, where either $y^2 + y = f(x)$ (here
$f(x) \in K[x]$ has degree $3$) or $y^2 + y  = x + \frac{1}{ax+b}$ with $a,b\in K$ and $a\ne 0$.
 
Let $\calP$ denote the set of places of $K(E)$ of degree 1. There is a unique common pole of
$x$ and $y$ which we denote by $Q_\infty$. This place has degree 1 and so belongs to 
 $\calP$. 
 Define the map
$$\Phi: \calP \to \mathrm{Cl}^0(E)$$
by $\Phi(P) = [P-Q_\infty]$.
This map \cite[Proposition 6.1.7]{stich}  is a bijection and induces an (abelian) group structure on $\calP$: $P \oplus Q = \Phi^{-1}( \Phi(P) + \Phi(Q))$. The place $Q_\infty$ is the identity element of this group.
% and $P\oplus Q = R$ iff  the divisor $P+Q-R-Q_\infty$ is principal and in this case the Riemann-Roch space $\calL(P+Q-R-Q_\infty)$ is one dimensional. 
It follows that if $P$ and $Q$ are rational places, then $P-Q$ is a principal divisor if and only if $P=Q$. Thus the Riemann-Roch space $\calL(P-Q)$ has positive dimension if and only if  $P=Q$.

We need to distinguish between the operations of the group of divisors and the elliptic curve group law; and we also need to distinguish between places and their corresponding points on the elliptic curve. 
We do so as follows. Each place $P$ of $K(E)$ corresponds to a unique point  on the elliptic curve defined by any one of the above given equations. We denote the corresponding point in bold font $\P$.
Thus the sum $P+Q$ is a divisor of the function field $K(E)$ while  the sum $\P+\Q$ is really another point on $E$ according to  the elliptic curve group law.  Henceforth we assume that $K(E)= K(x,y)$ where $x$ and $y$ are related by any one of the defining equations above for an elliptic curve.

Suppose the degree 1 places of $K(E)$ are
$P_0,P_1,P_2, \ldots, P_{n-1}$ where $P_0 := Q_\infty$ is the unique common pole of $x$ and $y$. 
In accordance with the notation introduced above, $\mathcal{P}$ denotes the set of places $P_0,P_1,\ldots, P_{n-1}$.
For a place $P$, we denote by $P'$ the place of $K(E)$ corresponding to the additive inverse of
$\P$ (so $\P +  \P' = \Q_\infty$). Note that $x(\P) = x(\P')$.

We define $m(\P,\Q)$ to be the line through $\P$ and $\Q$ if both $P,Q\ne Q_\infty$, that is
$m(\P,\Q) = ax+by+c$ for some $a,b,c \in \mathbb{F}_q$ and the points $\P,\Q$ lie on this line. 
Note that if $Q=P$ ($\ne Q_\infty$) then $m(\P,\Q)$ is the tangent line to $E$ at the point $\P$.
If $Q = P'$ ($\ne Q_\infty$) then $m(\P,\Q) = x-x(\P) = x-x(\Q)$.
If $P= Q_\infty$ or $Q=Q_\infty$ then we define $m(\P,\Q) := 1 \in \mathbb{F}_q$.

If $P\ne Q_\infty$ and $Q\ne Q_\infty$  and $\P +  \Q =\R$ then
it is well known that $m(P,Q)$ has three points of intersection with 
the elliptic curve and  thus
$$( m(\P,\Q))  = P + Q  + R' - 3Q_\infty.$$
Here it is possible that $R' = Q_\infty$, in which case $Q=P'$.
If $Q=P'$, then
$$m(\P,\Q) = x-x(\P) = x-x(\Q)$$
and
$$ (m(\P,\Q)) = P + P' -2Q_\infty.$$
Thus, if $\P + \Q = \R$ and $R\ne Q_\infty$, it follows that
$$\left (  \frac{ m(\P,\Q)}{x-x(\R)} \right ) = P+Q-R-Q_\infty.$$

Suppose that $\P + \Q = \R$. Then we define the following function:
 $$F(\P,\Q) := \left \{ \begin{array}{ll}    
\frac{x-x(\R)}{m(\P,\Q)} & \mathrm{if} \, \P,\Q,\R\ne \Q_\infty \\
\frac{1}{m(\P,\Q)}    &  \mathrm{if} \, \P,\Q \ne\Q_\infty\, \mathrm{but} \, \R=\Q_\infty \\
1 & \mathrm{if}\, \P = \Q_\infty\, \mathrm{or}\, \Q = \Q_\infty.
\end{array} \right.$$
One easily checks in all three cases the divisor of $F(\P,\Q)$ is 
$$(F(\P,\Q)) =  -P-Q+R+Q_\infty.$$

We repeatedly use the result that if $D$ is a divisor and $f$ a function in an algebraic function field,
then $f\calL(D) = \calL(D-(f))$.

\begin{prop}\label{prop1}
Let  $P,Q$ be rational places of $K(E)$. Then for a rational place $R$ of $K(E)$, $\P  + \Q = \R$ if and only if $\calL(P+Q-R -Q_\infty) \ne 0$, in which case
$$\calL(P+Q-R-Q_\infty) = \span ( F(\P,\Q) ).$$
\end{prop}

\begin{proof}
The forward implication is obvious. For the reverse implication, let $R$ be a rational place of $K(E)$ and suppose that 
$$\calL(P+Q-R-Q_\infty)  \ne 0.$$
We need to show that $\P + \Q = \R$. First suppose that $P,Q \neq Q_\infty$, then
$$\frac{1}{F(\P,\Q)}\calL(P+Q-R-Q_\infty)  = \calL(S-R),$$
where $\S$ is the additive inverse of the  third point of intersection of the line $m(\P,\Q)$ with the elliptic
curve~$E$ (it may happen that $S  = Q_\infty$). Since $\calL(S-R)$ 
has positive dimension, it follows that  $R=S$ and
$\calL(S-R) = \mathbb{F}_q$ so that
$$\calL(P+Q-R-Q_\infty) = \span ( F(\P,\Q) ).$$

If $P = Q_\infty$, then $\P+\Q = \Q$ and 
$\calL(P+Q-R-Q_\infty) = \calL(Q-R)$ is nontrivial by assumption and it follows that $R=Q$ so $\R = \P + \Q$
and $\calL(P+Q-R-Q_\infty)  = \span \{1 \} = \span ( F(\P,\Q) )$.

Likewise, the reverse implication is true if $\Q=\Q_\infty$.
\end{proof}

\begin{thm} \label{npinfty}
For an integer $n\ge 1$, 
$n\P  = \Q_\infty$
if and only if 
$$ \calL(nP-nQ_\infty) = \span \{ F(\P,\P)F(\P,2\P)\ldots F(\P,(n-1)\P) \}.$$
\end{thm}

\begin{proof}
For $n=1$ this result is trivial so we assume that $n>1$. For $k \ge 1$, we put $\P_k := k\P$. Observe that for $k\ge 2$,
$\P + \P_{k-1} = \P_k$, and so 
$$\calL(P+ P_{k-1} -P_k - Q_\infty) = \span \{ F(\P,\P_{k-1})\}.$$
We will use this fact repeatedly. Suppose that $n\P = \Q_\infty$. Then $\P + \P_{n-1} = \Q_\infty$, whence
$$\calL(P + P_{n-1}-2Q_\infty) = \span \{ F(\P,\P_{n-1}) \}.$$ 
Since $\P+\P_{n-i-1} = \P_{n-i}$ for $i=1,3,\ldots n-2$, the following identities are true:

\begin{eqnarray*}
\calL(P+P_{n-2}-P_{n-1} - Q_\infty)  &  =   &    \span \{  F(\P,\P_{n-2}) \} \\
\calL(P+P_{n-3}-P_{n-2} - Q_\infty)  &  =   &    \span \{ F(\P,\P_{n-3}) \}  \\
  &\vdots&  \\
\calL(P+P-P_{2}-Q_\infty)              &       =   &   \span \{ F(\P,\P) \}.\\
\end{eqnarray*}

\noindent
Notice that  if $\calL(D_1) = \span \{ f_1 \}$ and  $\calL(D_2) = \span \{ f_2 \}$ then $\calL(D_1 + D_2) = \span \{ f_1f_2 \}$. Combining this observation with the above identities, we obtain
\begin{equation}
\label{LFP}
\calL(nP-nQ_\infty) = \span \{ F(\P,\P) F(\P,\P_2) \ldots F(\P,\P_{n-1}) \}.
\end{equation}

On the other hand, assume that~\eqref{LFP} holds. Since the divisor of
$$F(\P,\P) F(\P,\P_2) \ldots F(\P,\P_{n-2})$$
is 
\begin{eqnarray*}
(-P - P + P_2 + Q_\infty) + (-P - P_2 + P_3+Q_\infty) \\
+ \ldots + (-P - P_{n-2} + P_{n-1} +Q_\infty) \\
= -(n-1)P +P_{n-1} + (n-2)Q_\infty,
\end{eqnarray*}
we have that
$$\frac{1}{F(\P,\P) F(\P,\P_2) \ldots F(\P,\P_{n-2}) } \calL(nP-nQ_\infty)
= \calL( P+P_{n-1} - 2Q_\infty)$$
is nontrivial. By Proposition \ref{prop1} it follows that $\P+\P_{n-1} = \Q_\infty$, that is,
$n\P = \Q_\infty$, as required.
\end{proof}

\begin{thm}\label{gen}
Let 
$$D := rQ_\infty + \sum_{i=1}^{n-1} a_i P_i$$
be a divisor of degree $0$. Then $D$ is principal if and only if
$$\sum_{i=1}^{n-1} a_i \P_i = \Q_\infty.$$
If $D$ is principal, then $D=(f)$, where $f$ is the product of functions of the form
$F(\P,\Q)$ with $P,Q\in \calP$. The group $\calO_\calP^*$ is generated by the functions $F(\P,\Q)$ where $\P,\Q\in \calP$. 
Consequently, the lattice $L_\calP$ is generated by vectors of the form $P+Q-R-Q_\infty$ where $\P + \Q = \R$.
\end{thm}

\begin{proof}
We can assume without loss of generality that $a_i\ge 0$ for $1\le i \le n-1$.  Indeed, for a place $P\in \calP$ and integer $k \ge 2$, let 
$$T_k(\P) :=   F(\P,\P)F(\P,2\P)\ldots F(\P,(k-1)\P).$$
Suppose that $a_j<0$ and let $k_j$ be the order of the point $\P_j$. 
By Theorem \ref{npinfty}, the divisor of $T_{k_j}(P_j)$ is $-k_j P_j + k_j Q_\infty$. Therefore  
$$\left ( \frac{1}{T_{k_j} (\P_j)}  \right )^{\ell } \calL(D) = \calL( D'),$$
where
$$D' :=  (r-\ell k_j) Q_\infty + \sum_{i=1,i \ne j }^{n-1} a_i P_i + (a_j+\ell k_j) P_j$$
and $a_j + \ell k_j \ge 0$ for sufficiently large $\ell$. Moreover, $D'$ is a principal divisor if and only if $D$ is a principal divisor and 
$$ \sum_{i=1,i \ne j }^{n-1} a_i \P_i + (a_j+\ell k_j) \P_j
= \sum_{i=1 }^{n-1} a_i \P_i .$$

Now write
$$D = rQ_\infty + Q_1+Q_2 + \ldots + Q_t,$$
where repetitions among the $Q_i$'s are allowed and $t=-r$. Put 
$$S_i := Q_{t-i} + Q_{t-i+1} + \ldots + Q_t,\ \ \T_i := \Q_{t-i} + \Q_{t-i+1} + \ldots +  \Q_t.$$
In accordance with the notation above, $T_i$ is the place corresponding to the point~$\T_i$. 
Put
$$ f :=  F(\Q_{t-1},\Q_t) F(\Q_{t-2},\T_1)F(\Q_{t-3},\T_2)\ldots F(\Q_1,\T_{t-2}). $$
We claim that 
$$ \frac{1}{f} \calL(D) = \calL( -Q_\infty + T_{t-1}).$$
This follows from the fact that the divisor of the function $\frac{1}{f}$ is
\begin{eqnarray*}
&\ & (Q_{t-1} +  Q_t - T_1 - Q_\infty) + (Q_{t-2} + T_1  - T_2 - Q_\infty) \\
& +& (Q_{t-3}+ T_2 - T_3 - Q_\infty)  +  \ldots +(Q_1+ T_{t-2}  - T_{t-1} - Q_\infty) \\
& = & Q_t + Q_{t-1} + ... + Q_1    -T_{t-1} - (t-1)Q_\infty,
\end{eqnarray*}
and
$$D - (1/f) = -Q_\infty + T_{t-1}.$$
The result now follows, since the divisor $-Q_\infty + T_{t-1}$ is principal
if and only if $\T_{t-1} = \Q_\infty$, that is, $\Q_1 + \Q_2 + \ldots + \Q_t = \Q_\infty$. Furthermore, 
$-Q_\infty+T_{t-1}$ is principal if and only if $ \frac{1}{f} \calL(D) = \calL( -Q_\infty + T_{t-1}) = \span\{1\}$, that is,
$\calL(D) = \span\{ f\}$.

The remaining statement of the theorem now follows quickly. Note that each function $F(\P,\Q)$ has  its support in $\calP$, that is $F(\P,\Q) \in \calO_\calP^*$.  Further observe that  the set $\calO_\calP^*$ is the union of all $\calL(D)\setminus \{0\}$ where 
$D$ runs over all principal divisors with support in $\calP$. From the above, we see that 
$\calL(D)$ is the span of  products of functions of the form $F(\P,\Q)$ where $P,Q\in \calP$. This completes the proof.
\end{proof}
\bigskip

\section{Lattices from elliptic curves}
\label{elliptic} 

We are now ready to prove our main results. We first establish an explicit value for the minimum distance of the lattice~$L_{\calP}$ in case of elliptic curves.

\begin{lem}\label{form}
Suppose that $n\ge 4$. Then the minimum distance of $L_\calP$ is $2$ and 
the minimal vectors of $L_\calP$ are of the form
$P+Q-R-S$ where $P,Q,R,S\in \calP$ are distinct and $\P+\Q = \R+\S$. 
If $n=3$ then the minimum distance of $L_\calP$ is $\sqrt{6}$ and the minimal vectors are of the form $\pm(P+Q-2Q_\infty)$, $\pm(P -2 Q +Q_\infty)$ and $\pm(-2P + Q + Q1_\infty$
where $\calP = \{ P,Q,Q_\infty\}$.
\end{lem}

\begin{proof}
Since a divisor $P-Q$ is principal if and only if $P=Q$, it follows that $\deg f \neq 1$ for any $f \in K(E)$.  First assume that $n\ge 4$. Then there are two distinct points $\P,\Q$, both not equal to $\Q_\infty$, such that $\P \ne \Q'$. Hence $\P + \Q = \R$ where $\R\ne \P,\Q,\Q_\infty$. The divisor
of the function $F(\P,\Q)$ is $-P- Q +R+Q_\infty$, so $d(L_\calP) \le 2$. On the other hand,~\eqref{mind} guarantees that $d(L_\calP) \ge  2$. Thus $d(L_\calP) = 2$.

Now consider a minimal vector $v$ of $L_\calP$. Then $v$ must be of the form
$P+Q-R-S$ where $P,Q,R,S$ are distinct rational places. Note also that 
$P+Q-R-S$ is a principal divisor.
Suppose that $\P+\Q = \R_1$. Then $P+Q-R_1-Q_\infty$ is a principal divisor and so 
is $(P+Q-R_1-Q_\infty) - (P+Q-R-S) = R+S-R_1-Q_\infty$. From Proposition \ref{prop1} we see
that $\R+\S =\R_1$. Thus the minimal vectors of $L_\calP$ are of the form
$P+Q-R-S$ where $P,Q,R,S\in \calP$ are distinct and $\P+\Q = \R+\S$. 

Next assume that $n=3$. Then $\calP = \{ Q_\infty, P, Q\}$ where $\Q= 2\P$. 
The following are vectors of $L_\calP$:    $3P-3Q_\infty$, $3Q-3Q_\infty$,
 $2P-Q-Q_\infty$, $P-2Q+Q_\infty$. Thus if $a_1P+b_1Q+c_1Q_\infty$ is a lattice vector, then
so is $a_2P+ b_2Q+ c_2Q_\infty$ where $a_2 = a_1 \mod 3$ and $b_2 = b_1 \mod 3$ and 
$c_2 = -a_2-b_2$.
One easily checks that  the only possibilities for the minimum vectors are
 $\pm(P+Q-2Q_\infty)$, $\pm(P -2 Q +Q1)$ and $\pm(-2P + Q + Q1)$, so $d(L_\calP) = \sqrt{6}$.
\end{proof}
\smallskip

Next we prove a formula for the number of minimal vectors in~$L_{\calP}$.

\begin{thm} \label{min_num}
Assume that $n\ge 4$ and let $\epsilon$ denote the number of 2-torsion points of $E$. 
Then the number of minimal vectors in $L_\calP$ is 
\begin{equation}\label{numminvecs}
\frac{n}{\epsilon} \cdot \frac{(n-\epsilon)(n-\epsilon-2)}{4} + 
\left (n-\frac{n}{\epsilon} \right)\cdot \frac{n(n-2)}{4}.
\end{equation}
\end{thm}
\begin{proof}
Define the homomorphism $\tau:E \to E$ by $\tau(\P) = 2\P$. Then the kernel of $\tau$ is the set
$E[2]$ of 2-torsion points of $E$ and the image of $\tau$  has $n/\epsilon$ points.  

Fix a point $\A$ of $E$. First we count the number of solutions to the equation
$\P + \Q = \A$ where $\P,\Q $ are distinct points of $E$.  Observe that 
$\P = \Q$ if and only if $\A \in \mathrm{Image}(\tau)$. 

If $\A \in \mathrm{Image}(\tau) $ there are $\epsilon$ solutions $\P$ to $2\P = \A$.  
Thus there are $n-\epsilon$ possible points $\P$ such that
$\Q := \A-\P \ne \P$, and so there are $(n-\epsilon)/2$ pairs $\P,\Q$ such that $\P + \Q = \A$
and $\P \ne \Q$. Hence the number of pairs $\R,\S$, disjoint from $\{\P,\Q\}$, such that $\R + \S = \A$, is $(n-\epsilon -2)/2$.
In total, there are $(n-\epsilon)/2 \cdot (n-\epsilon -2)/2 = (n-\epsilon)(n-\epsilon-2)/4$
possible minimal vectors $P+Q-R-S$ such that $\P+\Q = \A = \R + \S$.
The size of the image of $\tau$ is $\frac{n}{\epsilon}$  so the total number of
possible minimal vectors $P+Q-R-S$ such that $\P+\Q = \A = \R + \S$ with $\A \in \mathrm{Image}(\tau)$ is
$\frac{n}{\epsilon} \cdot \frac{(n-\epsilon)(n-\epsilon-2)}{4}.$

 If $\A  \not \in \mathrm{Image}(\tau) $ there are no solutions $\P$ to $2\P = \A$.  
Then similar reasoning as above shows that there are 
$(n-\frac{n}{\epsilon})\cdot \frac{n(n-2)}{4}$  minimal vectors $P+Q-R-S$ with
$\P + \Q   \not \in \mathrm{Image}(\tau) $. Thus by the above argument and Lemma~\ref{form}, the number of minimal vectors of $L_\calP$ is given by (\ref{numminvecs}).
\end{proof}
\smallskip

We are now ready to prove our main result, which is just a restatement of Theorem~\ref{main}.

\begin{thm} \label{main-1}
Suppose that $E$ has at least 5 points. Then the lattice $L_\calP$ is generated by its minimal vectors. In particular, this means that it is well-rounded.
\end{thm}

\proof
We know from Theorem \ref{gen} that the lattice $L_\calP$ is generated by nonzero vectors of the form 
$v := -P-Q+R+Q_\infty$ where $\P + \Q = \R$.  It suffices to show that each such vector can be
written in terms of minimal vectors. Suppose that $v$ is not a minimal vector, that is, suppose that
$P,Q,R,Q_\infty$ are not all distinct. 
Notice that, since $v$ is a nonzero principal divisor, it cannot happen that $P$ or $Q$ equals~$Q_\infty$.  Similarly, it also cannot happen that  $P=R$  or $Q=R$. Thus one of the following must be true: $P=Q$ or $R=Q_\infty$.

Suppose that  $P=Q$. Then $v = -2P + R + Q_\infty$  and  $2\P = \R$. 
Since $E$ has at least five points, we can choose a rational place $U$ such that $\U$ is not any of 
$ \Q_\infty,\P, 2\P$ or $-\P$.
Put $\S := \P + \U$ and 
observe that 
$$
-2P + R + Q_\infty
=(-P- U + S+ Q_\infty) - (P + S - R - U)$$
We claim that $-P- U + S+ Q_\infty$ and  $P + S - R - U$ are minimal vectors.

By choice $U \ne P, Q_\infty$. Also $U\ne S$ otherwise $\P = \Q_\infty$.
Further, $S\ne P$ otherwise $U=Q_\infty$. Finally, $S\ne Q_\infty$ otherwise
$\P + \U = \Q_\infty$ whence $\U = -\P$, which is not true.  Thus $-P-U+S+Q_\infty$ is a minimal vector.

Observe that $\P + \S = 2\P + \U  = \R+\U$  so $P+S-R-U$ is a lattice vector.
We already know that $S,P,U$ are distinct. We must show that none of $S,U$ equals $R$ 
(we pointed out above that $R\ne P$). If $R = S$ then $\U=\P$, which is not possible. 
If $R=U$ then $\U = \R = 2\P$, which is not possible. Thus $P+S-R-U$ is a minimal vector. 

We have shown that if $P=Q$, then $v$ is the difference of two minimal vectors.

Next assume that $R=Q_\infty$, so $v= -P-Q+2Q_\infty$ and $\P + \Q = \Q_\infty$.
Since $E$ contains at least 5 rational places, we can choose a rational point $\U$ different from the 
points  $\Q_\infty, \P, \Q, 2\P$.
Put $\S := \Q + \U$ and note that $U\ne S$ otherwise $Q=Q_\infty$.
Also $Q+U-S-Q_\infty$ is a lattice point
and $\P + \S = \P + \Q + \U = \U$ so that   $P+S-U-Q_\infty$ is also a  lattice point.
Now
$$v = P+Q-2Q_\infty = (Q+U-S-Q_\infty) + (P + S-U-Q_\infty)$$
is the sum of two lattice points. 
We claim that $Q+U-S-Q_\infty$ and  $P + S-U-Q_\infty$ are minimal vectors.

First we show that $Q+U-S-Q_\infty$ is a minimal vector. We must show that the places
$Q,U,S,Q_\infty$ are distinct.
By our choice $U\ne Q,Q_\infty$. We already pointed out that $U\ne S$. It is not possible for
$Q=S$, for otherwise $\Q= \S = \Q + \U$ whence $\U = \Q_\infty$, that is $U = Q_\infty$ thus contradicting our choice of $\U$.    Finally, $S\ne Q_\infty$, otherwise
$\U=-\Q = \P$. Thus $Q+U-S-\Q_\infty$ is a minimal vector. 

Next we show that $P + S-U-Q_\infty$ is a minimal vector.  From the argument above we know that
$S,U,Q_\infty$ are distinct.  Since $d(L_\calP) = 2$, we also  know that $P\ne U$ and $P\ne Q_\infty$. If $P=S$ then 
$\U = \P + \S = 2\P$, which is not possible by the choice of $U$. Thus $P + S-U-Q_\infty$
is a minimal vector. 

We have shown that $v$ is the sum of  two minimal vectors, which completes the proof of the theorem.
\endproof
\smallskip

Finally, we provide an estimate on the covering radius of~$L_{\calP}$. Recall that the covering radius (also called the inhomogeneous minimum) $\mu(L)$ of a lattice~$L$ is defined as
$$\mu(L) = \inf \left\{ r \in \real_{>0} : B_V(r) + L = V \right\},$$
where $V = \spn_{\real} L$ and $B_V(r)$ is the closed ball of radius $r$ centered at the origin in~$V$. In addition to an estimate on~$\mu(L_{\calP})$, our next theorem can also be interpreted as a result about the closest vector problem on such lattices.

\begin{thm} \label{cover} The covering radius of~$L_{\calP}$ satisfies the inequality
\begin{equation}
\label{cvr_bound}
\mu(L_{\calP}) \leq \frac{1}{2} \left( \sqrt{    n^2 +4n + 8 }    +  \sqrt{n}  \right).
\end{equation} 
In other words, if $V = \spn_{\real} A_{n-1} = \spn_{\real} L_{\calP} \subset \real^n$ and $v \in V$, then there exists a lattice point in $L_\calP$ within distance $\frac{1}{2} \left ( \sqrt{    n^2 +4n + 8 }    +  \sqrt{n}  \right )$ from $v$. Furthermore, if $v \in A_{n-1}$ then there is a lattice point in $L_\calP$ within distance $\sqrt{2}$ from $v$.
\end{thm}

\begin{proof}
Suppose that $v := (r_0,r_1,\ldots, r_{n-1})$ is a point in $V$, so $r_0 + ... + r_{n-1} = 0$.
Let $w_1 := (a_0,a_1,\ldots, a_{n-1}) \in \ZZ^n$ where $a_i$ is the nearest integer
to $r_i$ (note that if $r_i$ is a half integer, then $a_i$ is just the floor of $r_i$). 
Now
$a_0\P_0 + a_1\P_1 + \ldots + a_{n-1} \P_{n-1}$ equals a point $\P_j$ for some $j$, $0\le j \le n-1$. 

First suppose that $j\ne 0$.
Put $A_0 :=-a_1-a_2-\ldots -a_{n-1} +1$. Then by Theorem \ref{gen},  the vector  $w_2 := (A_0,a_1, \dots,     a_{j-1}, a_j-1, a_{j+1}, \ldots, a_{n-1})$
is a lattice point and the distance between $w_2$ and $v$ is 
\begin{eqnarray}
&     & ||v-w_2||   \nonumber                                                    \\                   
&\le & ||v-w_1|| + ||w_1-w_2||                   \nonumber          \\
&\le & \sqrt{n/4} + \sqrt{ (A_0-a_0)^2  + 1}      \nonumber     \\
&=   & \sqrt{n/4} + \sqrt{ (a_0+a_1 + \ldots + a_{n-1} -1)^2  + 1}        \nonumber  \\
&=   & \sqrt{n}/2 + \sqrt{   S^2 -2S + 2 },   \label{eqn4}
\end{eqnarray}
where $S = a_0 + a_1 + \ldots + a_{n-1}$. Now
\begin{eqnarray*}
|S| & = & |a_0 + a_1 + \dots + a_{n-1} |  =  |(a_0-r_0) + .... + (a_{n-1}-r_{n-1})| \\
& \le & |a_0-r_0| + ... |a_{n-1}-r_{n-1}| \le n/2.
\end{eqnarray*}
Thus
$$||v-w_2|| \le \sqrt{n}/2 + \sqrt{    n^2/4 +2(n/2) + 2 }
= \frac{1}{2} \left (  \sqrt{n} + \sqrt{    n^2 +4n + 8 }  \right),$$
as required. 

If $j=0$, put 
$A_0 = -a_1 - a_2 -\ldots -a_{n-1}$ and $w_2 = (A_0,a_1,\ldots, a_{n-1})$. Then
$$||v-w_2|| \le  ||v-w_1|| + ||w_1-w_2||  \le 
\sqrt{n/4} + \sqrt{ (A_0-a_0)^2} = 
\sqrt{n}/2 + |S|
\le \sqrt{n}/2  + n/2,$$
by the argument above. This is still less than the claimed bound (\ref{cvr_bound}).

The remaining assertion of the theorem easily follows
from the above argument: if $v\in A_{n-1}$ then $S=0$ and $v=w_1$, so from (\ref{eqn4}) we obtain that 
$||v-w_2|| \le \sqrt{2}$, as claimed.
\end{proof}

\begin{remark} Suppose that $n \geq 5$, then~$L_{\calP}$ is well-rounded by Theorem~\ref{main-1}, and $d(L_{\calP})=2$ by Lemma~\ref{form}. In this case, the standard bounds on covering radius of a lattice (see~\cite{gruber}) guarantee that
$$\mu(L_{\calP}) \leq n-1,$$
which is weaker than our bound~\eqref{cvr_bound} when $n$ is sufficiently large (this would imply that $q$ is also large, since $n \leq q+1+2\sqrt{q}$ by Hasse's theorem).
\end{remark}

\section{Acknowledgement}

The authors would like to thank  Dr Min Sha for his indepth reading, corrections and comments on the previous version of the manuscript. He also pointed out that Theorem 3.3 is still true if the elliptic curve has fewer than 5 points.

%\nocite{*}
\bibliographystyle{plain}  % Here the bibliography 
\bibliography{elliptic_wr}    % is inserted.

\begin{thebibliography}{1}

\bibitem{cohn_kumar}
H.~Cohn and A.~Kumar.
\newblock Optimality and uniqueness of the {L}eech lattice among lattices.
\newblock {\em Ann. of Math. (2)}, 170(3):1003--1050, 2009.

\bibitem{conway}
J.~H. Conway and N.~J.~A. Sloane.
\newblock {\em Sphere Packings, Lattices, and Groups}.
\newblock Springer-Verlag, 3rd edition, 1999.

\bibitem{gruber}
P.~M. Gruber and C.~G. Lekkerkerker.
\newblock {\em Geometry of numbers}.
\newblock North-Holland Publishing Co., 2nd edition, 1987.

\bibitem{martinet}
J.~Martinet.
\newblock {\em Perfect Lattices in Euclidean Spaces}.
\newblock Springer-Verlag, 2003.

\bibitem{stich}
H.~Stichtenoth.
\newblock {\em Algebraic Function Fields and Codes}.
\newblock Springer, Berlin, 2nd edition, 2009.

\bibitem{tv}
M.~A. Tsfasman and S.~G. Vladut.
\newblock {\em Algebraic-Geometric Codes}.
\newblock Kluwer Academic Publishers, 1991.

\end{thebibliography}
\end{document}